\documentclass[reqno,  a4paper, 11pt]{amsart}
\usepackage[colorlinks=true]{hyperref}
\usepackage[english]{babel}
\usepackage{amsmath}
\usepackage{amsfonts}
\usepackage{thmtools}
\usepackage{amssymb}
\usepackage{tikz-cd}
\usepackage{amsthm}
\usepackage{todonotes}
\usepackage{stmaryrd}
\usepackage{microtype}
\usepackage{mathrsfs}
\usepackage{enumitem}

\usepackage[nameinlink]{cleveref}

\newcount\hh
\newcount\mm
\mm=\time
\hh=\time
\divide\hh by 60
\divide\mm by 60
\multiply\mm by 60
\mm=-\mm
\advance\mm by \time
\def\hhmm{\number\hh:\ifnum\mm<10{}0\fi\number\mm}

\DeclareMathOperator{\tame}{tame}

\DeclareMathOperator{\codim}{codim}

\DeclareMathOperator{\Spec}{Spec}

\DeclareMathOperator{\bSpec}{\mathbf{Spec}}

\DeclareMathOperator{\Spf}{Spf}

\DeclareMathOperator{\Char}{char}

\DeclareMathOperator{\Leff}{Leff}
\DeclareMathOperator{\Lef}{Lef}
\DeclareMathOperator{\Div}{div}

\DeclareMathOperator{\Hom}{Hom}

\DeclareMathOperator{\red}{red}

\DeclareMathOperator{\Supp}{Supp}

\newcommand{\ab}{\text{ab}}

\DeclareMathOperator{\et}{\text{\'et}}

\newcommand{\Z}{\mathbb{Z}}

\newcommand{\Q}{\mathbb{Q}}

\renewcommand{\subset}{\subseteq}

\DeclareMathOperator{\Rev}{Rev}
\DeclareMathOperator{\RevEt}{RevEt}
\renewcommand{\phi}{\varphi}

\declaretheoremstyle[spaceabove=15pt,spacebelow=15pt, qed=$\square$]{defstyle}
\declaretheoremstyle[spaceabove=15pt, spacebelow=15pt, bodyfont=\itshape, qed=$\square$]{thmstyle}
\declaretheorem[style=thmstyle,numberwithin=section]{Theorem}
\declaretheorem[sibling=Theorem, style=thmstyle]{Proposition}
\declaretheorem[sibling=Theorem,  style=thmstyle]{Lemma}
\declaretheorem[sibling=Theorem, style=thmstyle]{Corollary}

\declaretheorem[sibling=Theorem, style=defstyle]{Definition}

\declaretheorem[sibling=Theorem, style=defstyle]{Remark}

\setenumerate{label={(\alph*)}, ref=(\alph*)}

\title{Lefschetz theorems for tamely ramified coverings}
\author{H\'el\`ene Esnault}
\address[H\'el\`ene Esnault]{Freie Universit\"at Berlin, FB Mathematik und Informatik, Arnimallee 3,
	14195 Berlin, Germany}
\email{esnault@math.fu-berlin.de}
\author{Lars Kindler}
\address[Lars Kindler]{Harvard University,
	Department of Mathematics,
	Science Center,
	One Oxford Street ,
	Cambridge, MA 02138, USA}
\email{kindler@math.harvard.edu}

\thanks{The second author gratefully acknowledges financial support of ``Deutsche Forschungsgemeinschaft'' through a Forschungsstipendium.}
\date{\today{ }\hhmm}

\begin{document}
\begin{abstract}
As is well known, the Lefschetz theorems for the \'etale fundamental group of \cite[Ch.~V]{SGA1} do not hold. We fill a small gap in the literature showing they do for tame coverings. Let $X$ be a regular projective variety over a field $k$, and let $D\hookrightarrow X$ be a strict normal crossings divisor (\cref{defn:ncd}). Then, if $Y$ is an ample regular hyperplane  intersecting $D$ transversally, the restriction functor from tame \'etale coverings (\cref{defn:categories}) of $X\setminus D$ to those of $Y\setminus D\cap Y$ is an equivalence if dimension $X \ge 3$, and fully faithful if  
dimension $X=2$.
The method is dictated by \cite{Grothendieck/Tame}. The authors showed that one can lift tame coverings from $Y\setminus D\cap Y$ to the complement of $D\cap Y$ in the formal completion of $X$ along $Y$. One has then to further lift to $X\setminus D$.
\end{abstract}

\maketitle
\section{Introduction}
Let $X$ be a locally noetherian scheme, let $Y$ be a closed subscheme, and let  $X_Y$ be the formal completion of $X$ along $Y$.  Recall (see  \cite[X.2, p.~89]{SGA2}) that the condition $\Lef(X,Y)$ holds if for every open neighborhood $U$  of $Y$ and every coherent locally free sheaf $E$ on $U$, the canonical map $H^0(U,E)\rightarrow H^0(X_Y,E_{X_Y})$ is an isomorphism.  For the condition $\Leff(X,Y)$, one  requires in addition that every coherent locally free sheaf on $X_Y$ is the restriction of a coherent locally free sheaf on some open neighborhood $U$ of $Y$.

Assume $X$ is defined over a field $k$ and is proper.   Let $D$ be another divisor, which has no common component with $Y$, such that $D\hookrightarrow X$ and $D\cap Y\hookrightarrow Y$ are strict   normal crossings divisors (\cref{defn:ncd}).  Let   $\bar{y} \to Y\setminus D\cap Y$ be a geometric point.  We define the functoriality morphism 
\begin{equation}\label{eq:main-morphism}
		\pi_1^{\tame}(Y\setminus D,\bar{y})\rightarrow \pi_1^{\tame}(X\setminus D,\bar{y}).
\end{equation}
between the tame fundamental groups \cite[\S.7]{Kerz/Tameness}. If $\Char(k)=0$, this is the usual functoriality morphism between the \'etale fundamental groups of $Y\setminus D$ and $X\setminus D$.
The aim of this note is to prove:

\begin{Theorem}\label{thm:main}
Assume   $X$ and $Y$ are  projective, regular, connected over $k$.
\begin{enumerate}[label=\emph{(\alph*)}, ref=(\alph*)]
	\item\label{main:LEF} If   $\Lef(X,Y)$ holds, then \eqref{eq:main-morphism} is surjective.
	\item\label{main:LEFF} If $\Leff(X,Y)$ holds, and if $Y$ intersects all effective divisors on $X$, then \eqref{eq:main-morphism} is an isomorphism.
\end{enumerate}
\end{Theorem}
This generalizes Grothendieck's Lefschetz Theorem  \cite[X, Cor.~2.6, Thm.~3.10, p.~97]{SGA2} for $D=0$, which is then true under less restrictive assumptions.

As is well understood, the hypothesis of \cref{thm:main}, \ref{main:LEF} is satisfied if $\dim X \geq 2$ while those of \cref{thm:main}, \ref{main:LEFF} are satisfied if $\dim X>2$ and if $Y$ is an ample regular hyperplane in $X$ (see {\cite[X, Ex.~2.2, p.~92]{SGA2}}). 

We finally remark that if  $k$ is algebraically closed, an alternative approach to prove \cref{thm:main}, \ref{main:LEFF} would be through the theory of regular singular stratified bundles by combining \cite[Thm.~5.2]{Gieseker/FlatBundles} with \cite[Thm.~1.1]{Kindler/FiniteBundles}. 

\medskip

If $X$ is no longer regular, but has singularities in codimension $\ge 2$, if $D\hookrightarrow X$ is  a divisor such that $X\setminus D$ is smooth, and if $k$ is finite, Drinfeld \cite[Cor.~C.2, Lemma C.3]{Drinfeld} proved that if $Y\hookrightarrow X$ is a regular projective curve,  intersecting the smooth locus of $D$ transversally, then the restriction functor from \'etale covers of $X\setminus D$, tamely ramified along the smooth part of the components of $D$, to the one on $Y\setminus D\cap Y$, is fully faithful.  By standard arguments we show in \cref{prop:drinfeld} that one may assume $k$ to be any field.
\medskip

As one does not have  at disposal resolution of singularities in characteristic $p>0$, it would be nice to generalize Drinfeld's theorem from  dimension $Y$ equal to $ 1$ to higher dimension, even if over a imperfect field one has to assume $X\setminus D$ to be smooth. However it is not even clear what would then be the correct formulation. In another direction, in light of Deligne's finiteness theorem \cite{Esnault/Kerz}, one would like to define a good notion of fundamental group with bounded ramification and show Lefschetz' theorems  for it. The abelian quotient of this yet non-existing theory is the content of 
\cite{Kerz/Lefschetz}. 

\medskip

In \Cref{sec:tameness} we make precise the notions of tame coverings and normal crossings divisor that we use.  In \Cref{sec:tameness-formal} we recall Grothendieck-Murre's   notion of tameness for finite maps between formal schemes and prove the first important lemma (\Cref{lemma:excellent}), before we carry out the proof of  \Cref{thm:main} in \Cref{sec:proof}.  In \Cref{sec:drinfeld} we extend Drinfeld's theorem over any field. We comment in \Cref{sec:generalization} on the relation between the Lefschetz' theorems discussed in this note and Deligne's finiteness theorem. 

\medskip

\emph{Acknowledgements:}  It is a pleasure to thank Moritz Kerz for discussions on the topic of this note at the time \cite{Esnault/Kerz} was written.
We then posed the problem solved in this note  to Sina Rezazadeh, who unfortunately left mathematics.

\section{Tamely ramified coverings}\label{sec:tameness}
We recall the definition of a (strict) normal crossings divisor.
\begin{Definition}[{\cite[1.8, p.~26]{Grothendieck/Tame}}]\label{defn:ncd}
Let $X$ be a locally noetherian scheme, and let $\{D_i\}_{i\in I}$ be  a finite set of effective Cartier divisors on $X$. For every $x\in X$ define $I_x:=\{i\in I| x\in \Supp D_i\}\subset I$.
%
%
\begin{enumerate}[label={(\alph*)}]
\item  The family of divisors $\{D_i\}_{i\in I}$ is said to have \emph{strict normal crossings} if for every $x\in \bigcup_{i\in I}\Supp(D_i)$ 
\begin{enumerate}[label={(\roman*)}]
	\item the local ring $\mathcal{O}_{X,x}$ is regular;
	\item for every $i\in I_x$, locally in $x$ we have $D_i=\sum_{j=1}^{n_i} \Div(t_{i,j})$ with $t_{i,j}\in \mathcal{O}_{X,x}$, such that the set $\{t_{i,j}|i\in I_x, 1\leq j\leq n_i\}$ is part of a regular system of parameters of $\mathcal{O}_{X,x}$.
\end{enumerate}
\item The family of divisors $\{D_i\}_{i\in I}$ is said to have \emph{normal crossings} if every $x\in \bigcup_{i\in I}\Supp(D_i)$ has an \'etale neighborhood $\gamma:V\rightarrow X$, such that the family $\{\gamma^*D_i\}_{i\in I}$ has strict normal crossings.
\item An effective Cartier divisor $D$ hast (strict) normal crossings if the underlying family of its reduced irreducible components has (strict) normal crossings.  
\end{enumerate}
\end{Definition}
\begin{Remark}\label{rem:ncd}
The divisor $D$ has strict normal crossings if and only if it has normal crossings and if its irreducible components are regular. One direction is \cite[Lemma 1.8.4, p.~27]{Grothendieck/Tame}, while the other direction comes from  (a) (ii), as the $t_{i,x}\in \mathcal{O}_{X,x}$ are local parameters. 
\end{Remark}

\begin{Definition}\label{defn:categories}Let $X$ be a locally noetherian, normal scheme and let $D$ be a divisor on $X$ with normal crossings.  We write $\Rev(X)$ for  the category of all finite $X$-schemes and  $\RevEt(X)$ for  the category of finite \'etale $X$-schemes.  Following \cite[2.4.1, p.~40]{Grothendieck/Tame}, we define $\Rev^D(X)$ to be the full subcategory of $\Rev(X)$ with objects the finite $X$-schemes  tamely ramified along $D$. Recall that a finite morphism $f:Z\rightarrow X$ is called \emph{tamely ramified} along $D$, if
\begin{enumerate}[label={(\roman*)}]
\item $Z$ is normal,
\item $f$ is \'etale over $X\setminus \Supp(D)$,
\item every irreducible component of $Z$ dominates an irreducible component of $X$,
\item for $x\in D$ of codimension $1$ in $X$, and any $z\in Z$  mapping to $x$, the extension of discrete valuation rings $\mathcal{O}_{X,x}\rightarrow \mathcal{O}_{Z,z}$ is tamely ramified (\cite[Def.~2.1.2, p.~30]{Grothendieck/Tame}).
\end{enumerate}
The natural functors  $\RevEt(X) \to \Rev^D(X)\to \Rev(X)$ are fully faithful. 
\end{Definition}
\begin{Remark}\label{rem:X-minus-D-vs-X}
The restriction functor $\Rev^{D}(X)\rightarrow \RevEt(X\setminus D)$ is fully faithful when $X$ is proper. Its essential image is the full subcategory of \'etale coverings of $X\setminus D$ which are tamely ramified along $D$, which, by the fundamental theorem \cite[Prop.~4.2]{Kerz/Tameness}, does not depend on the choice of $X$ and is even definable on a normal compactification of $X\setminus D$. 
A quasi-inverse functor  assigns to $Z\rightarrow X\setminus D$, \'etale, tame, with $Z$ connected,  the  normalization of $X$ in the function field of $Z$.	
\end{Remark}
\Cref{rem:X-minus-D-vs-X} shows that \cref{thm:main} is equivalent to the following.
\begin{Theorem}\label{thm:main-reformulation}
Let $k$ be a field, let $X$ be a projective, regular, connected $k$-scheme, let $D$ be a strict normal crossings divisor on $X$ and let $Y\subset X$ be a regular, closed subscheme, such that the inverse image $D|_Y$ of $D$ on $Y$ exists and is a strict normal crossings divisor.
\begin{enumerate}[label=\emph{(\alph*)}, ref=(\alph*)]
\item  If $\Lef(X,Y)$ holds then restriction induces a fully faithful functor
	\begin{equation}\label{eq:res-functor}\Rev^{D}(X)\rightarrow
		\Rev^{D|_Y}(Y).\end{equation}
\item If $\Leff(X,Y)$ holds and if $Y$ intersects every effective divisor on $X$, then
	\eqref{eq:res-functor} is an equivalence.
\end{enumerate}
\end{Theorem}
\section{Tamely ramified coverings of formal schemes}\label{sec:tameness-formal}
We recall a few definitions from \cite[\S.3, \S.4]{Grothendieck/Tame}.
\begin{Definition}[{\cite[3.1.4, 3.1.5, p.~45]{Grothendieck/Tame}}]
Let $\mathfrak{X}$ be a locally noetherian formal scheme.  If $D$ is an effective divisor on $\mathfrak{X}$ (that is, defined by an invertible coherent sheaf of ideals in $\mathcal{O}_{\mathfrak{X}}$, \cite[\S.21]{EGA4}), then  for any point $x\in \Supp(D)$, the localization $D_x$ is an effective divisor  on $\Spec \mathcal{O}_{\mathfrak{X},x}$.  The divisor $D$ is said to have \emph{(strict) normal crossings} (resp.~to be \emph{regular}) if $D_x$ is a (strict) normal crossings divisor (resp.~is a regular divisor) on $\Spec \mathcal{O}_{\mathfrak{X},x}$ for all $x\in \Supp(D)$.  A finite set $\{D_i\}_{i\in I}$ of effective divisors on $\mathfrak{X}$ is said to have (strict) normal crossings, if for every $x\in \mathfrak{X}$ the family $\{(D_i)_x\}_{i\in I}$ has (strict) normal crossings.
\end{Definition}
%
%
\begin{Definition}[{\cite[3.2.2, p.~49]{Grothendieck/Tame}}]
A morphism $f:\mathfrak{Y}\rightarrow \mathfrak{X}$  between two locally noetherian formal schemes is an \emph{\'etale covering}  if $f$ is finite, $f_*{\mathcal{O}_\mathfrak{Y}}$ is a locally free $\mathcal{O}_{\mathfrak{X}}$-module, and for all $x\in \mathfrak{X}$, the induced map of (usual) schemes $\mathfrak{Y}\times_{\mathfrak{X}}  \Spec k(x)\rightarrow  \Spec k(x)$ is \'etale.
As in the scheme case, we write $\Rev(\mathfrak{X})$ for the category of all finite maps to $\mathfrak{X}$ and $\RevEt(\mathfrak{X})$ for the category of all \'etale coverings of $\mathfrak{X}$.
\end{Definition}
\begin{Definition}[{\cite[4.1.2, p.~52]{Grothendieck/Tame}}] \label{defn:tame-formal}  
A locally noetherian formal scheme $\mathfrak{X}$ is said to be \emph{normal} if all stalks of $\mathcal{O}_{\mathfrak{X}}$ are normal. Let $\mathfrak{X}$ be normal and let $D$ be a divisor with normal crossings on $\mathfrak{X}$. A finite morphism $f:\mathfrak{Y}\rightarrow \mathfrak{X}$ is said to be  a \emph{tamely ramified covering with respect to $D$}, if for every $x\in \mathfrak{X}$ the finite morphism of schemes
\[\Spec((f_*\mathcal{O}_{\mathfrak{Y}})_x)\rightarrow \Spec(\mathcal{O}_{\mathfrak{X},x})\]
is tamely ramified along the normal crossings divisor $D_x$ in $\Spec \mathcal{O}_{\mathfrak{X},x}$.

We write $\Rev^D(\mathfrak{X})$ for the category of tamely ramified coverings of $\mathfrak{X}$ with respect to $D$.
\end{Definition}

The first main ingredient in the proof of \cref{thm:main-reformulation} is the following lifting result.
\begin{Theorem}[{\cite[Thm.~4.3.2, p.~58]{Grothendieck/Tame}}]\label{formal-lifting}
Let $\mathfrak{X}$ be a locally noetherian, normal formal scheme  and let $(D_i)_{i\in I}$ be a finite set of regular divisors  with normal crossings on $\mathfrak{X}$. Write $D:=\sum_{i\in I} D_i$. Let $\mathcal{J}$ be an ideal of definition of $\mathfrak{X}$ with the following properties.
\begin{enumerate}[label=\emph{(\alph*)}, ref=(\alph*)]
	\item\label{formal-lifting-cond1} The scheme $X_0:=(\mathfrak{X},\mathcal{O}_{\mathfrak{X}}/\mathcal{J})$ is normal;
	\item\label{formal-lifting-cond2} the inverse images $D_{i,0}$ of the divisors $D_i$ on $X_0$ exist, are regular, and the family $(D_{i,0})_{i\in I}$ has normal crossings. Write $D_0:=\sum_{i\in I}D_{i,0}$.
\end{enumerate}
Then the  restriction functor
\[
	\Rev(\mathfrak{X})\rightarrow \Rev(X_0), \ (\mathfrak{Z}\rightarrow \mathfrak{X})\mapsto (\mathfrak{Z}\times_{\mathfrak{X}}X_0\rightarrow X_0)
\] 
restricts to an equivalence of categories
\[\Rev^D(\mathfrak{X})\rightarrow \Rev^{D_0}(X_0).\]
\end{Theorem}
\section{Some facts about formal completion}\label{sec:completion}
The following facts are probably well-known, but we could not find a reference.

\begin{Lemma}\label{lemma:completion}
Let $A$ be an excellent ring and let $I\subset A$ be an ideal. Assume that $A^*:=\varprojlim_n A/I^n$ is excellent (see \cref{rem:excellenceOfCompletion}). Write $X:=\Spec A$, $Y:=\Spec A/I$ and $\mathfrak{X}:=\Spf A^*$. 	
Then $X$ is normal in some open neighborhood of $Y$ if and only if $\mathfrak{X}$ is normal.	
\end{Lemma}
\begin{Remark}\label{rem:excellenceOfCompletion}
As a special case of \cite[7.4.8, p.~203]{EGA4}, Grothendieck asks whether $A^*$ is excellent whenever $A$ is. O.~Gabber  has proved this result unconditionally (\cite[Remark 3.1.1]{Temkin/Desingularization}, \cite[Remark 1.2.9]{Kedlaya/Excellent}). Unfortunately, to our knowledge, the proof is not yet available in written form. 

On the other hand, it is proved in \cite{Valabrega} that if $A$ is a finitely generated algebra over a field, then $A^*$ is excellent. We shall  apply \cref{lemma:completion}  only in this situation.
\end{Remark}
In the sequel, the following lemma is crucially used. 
\begin{Lemma}[{\cite[7.8.3, (v), p.~215]{EGA4}}]\label{lemma:excellent} 
Let $(R,\mathfrak{m})$ be an excellent local ring and let $J\subset \mathfrak{m}$ be an  ideal. Then $R$ is normal if and only if the $J$-adic completion $\varprojlim_i R/J^i$ is normal.
\end{Lemma}

We prove the main result of this section.
\begin{proof}[{Proof of \cref{lemma:completion}}] 
We use the notations from the statement of \cref{lemma:completion}.  For a prime ideal $\mathfrak{p}\in \Spec A$ containing $I$, denote by $\mathfrak{p}^*$ the corresponding prime ideal in $A^*$ and also the corresponding point of $\mathfrak{X}=\Spf(A^*)$. Since the normal locus of $\Spec A$ is open (\cite[Scholie 7.8.3, (iv), p.~214]{EGA4}), we need to show that for a prime ideal $\mathfrak{p}\subset A$ containing $I$, the local ring $A_{\mathfrak{p}}$ is normal  if and only if $\mathcal{O}_{\mathfrak{X},\mathfrak{p}^*}$ is normal. 

Let $\mathfrak{p}\in \Spec A$ be a prime ideal containing $I$. The canonical map  of local rings $A_{\mathfrak{p}}\rightarrow A^*_{\mathfrak{p}^*}$ becomes an isomorphism $\widehat{A_\mathfrak{p}}\xrightarrow{\cong} \widehat{A^*_{\mathfrak{p}^*}}$ after completion with respect to the maximal ideals (\cite[24.B, D, p.~173]{Matsumura/CommutativeAlgebra}).  As both $A$ and $A^*$ are excellent by assumption, the same is true for the localizations $A_\mathfrak{p}$ and $A^*_{\mathfrak{p}^*}$. Thus, \cref{lemma:excellent}   applied to the local rings $A_\mathfrak{p}$ and $A^*_{\mathfrak{p}^*}$, with the topologies defined by their maximal ideals, yields that $A_\mathfrak{p}$ is normal if and only if $\widehat{A_{\mathfrak{p}}}\cong \widehat{A^*_{\mathfrak{p}^*}}$ is normal, if and only if  $A^*_{\mathfrak{p}^*}$ is normal.

Let $A^*_{\mathfrak{p}^*}\rightarrow (A^*_{\mathfrak{p^*}})^*$ be the $I$-adic completion of the localization $A^*_{\mathfrak{p}^*}$ of $A^*$ at $\mathfrak{p}^*$.  It factors
\[
	A^*_{\mathfrak{p}^*}\xrightarrow{\lambda}\mathcal{O}_{\mathfrak{X},\mathfrak{p}^*}\xrightarrow{\mu} (A^*_{\mathfrak{p}^*})^*,
\] 
with $\lambda$ and $\mu$ both faithfully flat (\cite[3.1.2, p.~44]{Grothendieck/Tame}). Indeed, for $f\in A^*$, write $S_f:=\{1,f,f^2,\ldots\}$, and $A_{\{f\}}$ for the $I$-adic completion of $S_f^{-1}A$. Then $\mathcal{O}_{\mathfrak{X},\mathfrak{p}^*}=\varinjlim_{f\not\in \mathfrak{p}^*}A_{\{f\}}$ (\cite[10.1.5, p.~182]{EGA1}). Faithful flatness of $\lambda$ (resp.~$\mu$) now follows from \cite[Ch.0, 6.2.3, p.~56]{EGA1} together with \cite[Ch.~0, 7.6.13, p.~74]{EGA1} (resp.~\cite[Ch.~0, 7.6.18, p.~75]{EGA1}).

We complete the proof: If $\mathcal{O}_{\mathfrak{X},\mathfrak{p}^*}$ is normal, then by faithfully flat descent $A^*_{\mathfrak{p}^*}$ is normal (\cite[21.E, p.~156]{Matsumura/CommutativeAlgebra}), and thus, as we saw above, $A_\mathfrak{p}$ is normal as well.
Conversely, if $A_\mathfrak{p}$ is normal, then the excellent ring $A^*_{\mathfrak{p}^*}$ is normal,  and so is its $I$-adic completion $(A^*_{\mathfrak{p}^*})^*$  (\cref{lemma:excellent}).  By faithfully flat descent, $\mathcal{O}_{\mathfrak{X},\mathfrak{p}^*}$ is normal as well.
%
\end{proof}

\begin{Corollary}\label{cor:formal-lifting-concrete}
Let $k$ be a field and let $X$ be a normal, separated, finite type $k$-scheme with $D\subset X$ a strict normal crossings divisor. Let $Y\subset X$ be a normal closed subscheme, such that the inverse image $D|_Y$ of $D$ on $Y$ exists and is a strict normal crossings divisor, and let $X_Y$ be the formal completion of $X$ along $Y$. Then
\begin{enumerate}[label=\emph{(\alph*)}, ref=(\alph*)]
	\item The formal scheme $X_Y$ is normal, the inverse image $D|_{X_Y}$ of $D$ on $X_Y$ exists and is a normal crossings divisor with regular components.
	\item\label{formal-lifting-concrete-1} The  functor  $\Rev^{D|_{X_Y}}(X_Y)\rightarrow \Rev^{D|_Y}(Y)$  of restriction is an equivalence. 
	\item\label{formal-lifting-concrete-2} If $\mathfrak{Z}\rightarrow X_Y$ is a tamely ramified covering with respect to $D|_{X_Y}$, then $\mathfrak{Z}$ is a normal formal scheme. 
\end{enumerate}
\end{Corollary}
\begin{proof}
(a) $X_Y$ is locally noetherian and normal, according to \cref{lemma:completion} (here we use the fact that $X$ is of finite type over a field).  By \cref{rem:ncd} the components $\{D_i\}_{i\in I}$ of $D$ are regular divisors. Thus, according to \cite[4.1.4, p.~53]{Grothendieck/Tame}, if $j:X_Y\rightarrow X$ is the canonical map of locally ringed spaces, then $(j^*D_i)_{i\in I}$ is a family of regular divisors with normal crossings on the formal scheme $X_Y$.
			
(b)  The condition \ref{formal-lifting-cond2} of \cref{formal-lifting} is then fulfilled, as we assume that $D|_Y$ is a strict normal crossings divisor.  Thus \cref{formal-lifting} applies and \cref{cor:formal-lifting-concrete}, \ref{formal-lifting-concrete-1} follows.
		
(c) Let $f:\mathfrak{Z}\rightarrow X_Y$ be a tamely ramified covering. To prove that $\mathfrak{Z}$ is normal, we may assume that $X=\Spec A$ and $Y=\Spec A/I$. Let $A^*$ be the $I$-adic completion of $A$, so that $X_Y=\Spf A^*$. Let $B$ be the finite $A^*$-algebra such that $\mathfrak{Z}=\Spf(B)$. As  $X_Y$ is normal, $A^*$ is also normal (\cite[3.1.3, p.~44]{Grothendieck/Tame}). We can apply \cite[Lemma 4.1.3, p.~52]{Grothendieck/Tame}, which says that the fact that $f$ is tamely ramified with respect to $D$ is equivalent to the fact that the induced map $\Spec B\rightarrow \Spec A^*$ is tamely ramified with respect to the divisor on $\Spec A^*$ corresponding to $D$. In particular, $B$ is normal. As in \cref{lemma:completion}, for every $z\in \mathfrak{Z}$, corresponding to a prime ideal $\mathfrak{p}\subset B$ containing $IB$, we have a sequence of faithfully flat maps \[B_\mathfrak{p}\rightarrow \mathcal{O}_{\mathfrak{Z},z}\rightarrow (B_\mathfrak{p})^*,\] where $(-)^*$ denotes $IB$-adic completion. As $A$ is of finite type over a field, $A$ is excellent, so $A^*$  is excellent (see \cref{rem:excellenceOfCompletion}), and hence so are the finite $A$-algebra $B$ and its localization $B_{\mathfrak{p}}$. \Cref{lemma:excellent} implies that $(B_\mathfrak{p})^*$ is normal, so $\mathcal{O}_{\mathfrak{Z},z}$ is normal as well.
\end{proof}

\section{Proof of \cref{thm:main}}\label{sec:proof}
We saw that \cref{thm:main} is equivalent to \cref{thm:main-reformulation}.\\

Let $X,Y, D$ be as in \cref{thm:main-reformulation}.  Denote by $X_Y$ the completion of $X$ along $Y$.  In \cref{cor:formal-lifting-concrete} we proved that $X_Y$ is a normal formal scheme.

Restriction gives a sequence of functors
\[
	\Rev(X)\rightarrow\Rev(X_Y)\rightarrow \Rev(Y).
\]
According to \cite[Cor.~4.1.4, p.~53]{Grothendieck/Tame} and \cref{cor:formal-lifting-concrete} this sequence restricts to
\[
	\Rev^D(X)\xrightarrow{F_1}\Rev^{D|_{X_Y}}(X_Y)\xrightarrow{F_2} \Rev^{D|_Y}(Y).
\]

We already saw in \cref{cor:formal-lifting-concrete} that $F_2$ is an equivalence. It remains to show that $F_1$ is fully faithful if $\Lef(X,Y)$ holds, and that $F_1$ is an equivalence if $\Leff(X,Y)$ holds and $Y$ meets very effective divisor on $X$.

The fact that enables us to use $\Lef(X,Y)$ and $\Leff(X,Y)$, which are conditions involving coherent locally free sheaves, is that tame coverings are flat. More precisely, an object $Z\rightarrow X$ of $\Rev^D(X)$  is a flat morphism according to \cite[Cor.~2.3.5, p.~39]{Grothendieck/Tame}, and an object $\mathfrak{Z}\rightarrow X_Y$ of $\Rev^{D|_Y}(X_Y)$ is a flat morphism of formal schemes (\cite[3.1.7, p.~39]{Grothendieck/Tame} together with \cite[4.1.3, p.~52]{Grothendieck/Tame}).

If $f:Z\rightarrow X$ is a tamely ramified cover with respect to $D$, then $f$ is flat, so $f_*\mathcal{O}_Z$ is a locally free $\mathcal{O}_X$-module of finite rank.  Morphisms in $\Rev^{D}(X)$ are thus defined by morphisms of $\mathcal{O}_X$-algebras which are locally free $\mathcal{O}_X$-modules. Assuming $\Lef(X,Y)$, this means that for every pair  of objects $Z,Z'\rightarrow X$ of $\Rev^{D}(X)$ the restriction map
\[\Hom_{X}(Z,Z')\xrightarrow{\cong}\Hom_{X_Y}(Z_Y,Z'_Y)\] 
is bijective. This shows that $F_1$ is fully faithful.

An object $f:\mathfrak{Z}\rightarrow X_Y$ of $\Rev^{D|_{X_Y}}(X_Y)$ is determined by the locally free $\mathcal{O}_{X_Y}$-algebra $f_*\mathcal{O}_{\mathfrak{Z}}$. Assuming $\Leff(X,Y)$, for every such object there exists an open subset $U\subset X$ containing $Y$ and a locally free sheaf $\mathcal{A}$ on $U$ such that $\mathcal{A}|_{X_Y}\cong f_*\mathcal{O}_{\mathfrak{Z}}$.  As $\Lef(X,Y)$ holds, we can lift the algebra structure from $f_*\mathcal{O}_{\mathfrak{Z}}$ to $\mathcal{A}$. 
Indeed, the global section of $(f_*\mathcal{O}_{\mathfrak{Z}} \otimes_{\mathcal{O}_{X_Y}}f_*\mathcal{O}_{\mathfrak{Z}})^\vee \otimes_{\mathcal{O}_{X_Y}}f_*\mathcal{O}_{\mathfrak{Z}}  $
defining the algebra structure lifts to a global section of  $(\mathcal{A} \otimes_{\mathcal{O}_U} \mathcal{A})^\vee \otimes_{\mathcal{O}_U} \mathcal{A}$, endowing $\mathcal{A}$ with an $\mathcal{O}_U$-algebra structure. 
Write $Z:=\bSpec \mathcal{A}$. We obtain a finite, flat morphism $g:Z\rightarrow U$  which restricts to $f:\mathfrak{Z}\rightarrow X_Y$.

According to \cref{cor:formal-lifting-concrete}, \ref{formal-lifting-concrete-2}, the formal scheme $\mathfrak{Z}$ is normal. We can identify $\mathfrak{Z}$ with the formal completion of $Z$ along the closed subset $g^{-1}(Y)$. As $Z$ is excellent, \cref{lemma:completion} implies that $Z$ is normal in an open neighborhood of $g^{-1}(Y)$.  Now the assumptions of \cite[Cor.~4.1.5, p.~54]{Grothendieck/Tame} are satisfied, from which follows that there is an open subset $V\subset U\subset X$ containing $Y,$ such that $g_V:Z\times_U V\rightarrow V$ is tamely ramified with respect to $V\cap D$. \Cref{lemma:codim2} shows that $g$ extends to an object of $\Rev^D(X)$ lifting $f$.

\begin{Lemma}\label{lemma:codim2}
Assume that $Y$ meets every effective divisor on $X$. If $U\subset X$ is an open subset containing $Y$, then restriction induces an equivalence
\begin{equation}\label{eq:codim2}
	\Rev^D(X)\xrightarrow{\cong}\Rev^{D\cap U}(U)
\end{equation}
\end{Lemma}
\begin{proof}
	By assumption $Y$ intersects every effective divisor on $X$, so
		\[\codim_X(X\setminus U)>1.\]
	Given a finite morphism $Z\rightarrow U$, tamely ramified over $U\cap D$, the normalization $Z_X\rightarrow X$ of $X$ in $Z$ is finite \'etale over $X\setminus D$, as $X$ is regular (``purity of the branch locus'') and tamely ramified over $D$.  This yields a quasi-inverse functor to the restriction functor \eqref{eq:codim2}.
\end{proof}

\begin{Remark} \label{rmk:char0}
If $k$ has characteristic $0$, then the quotient homomorphism $\pi_1 (X\setminus D) \to \pi_1^{\rm tame} (X\setminus D)$ is an isomorphism, and \cref{thm:main} applies as well. For the corresponding theorem for the topological fundamental group, thus when $k=\mathbb{C}$,  assuming $X$ smooth but not necessarily  a normal crossings compactification of $X\setminus D$, we refer to \cite[1.2, Remarks, p.~153]{GoreskyMacPherson}.  Of course, by the comparison isomorphisms, the topological theorem implies \cref{thm:main} (a). 
\end{Remark}

\section{Drinfeld's theorem} \label{sec:drinfeld}
\begin{Theorem}[{Drinfeld's theorem, \cite[Prop.~C.2]{Drinfeld}}] \label{thm:drinfeld}
Let $X$ be a geometrically irreducible projective variety over a finite field $k$, let $D\subset X$ be a divisor, and let $\Sigma\subset D$ be a closed subscheme of codimension $\ge 1$ in $D$, such that $X\setminus \Sigma$ and $D\setminus \Sigma$ are smooth. Then  any geometrically irreducible curve $Y\subset X$ which intersects $D$ in $D\setminus \Sigma\cap D$, and  is transversal to $D\setminus \Sigma$,  has the property that  the restriction to $Y\setminus D\cap Y$ of any finite \'etale connected cover of $X\setminus D$, which is tamely ramified along $D\setminus \Sigma$, is connected.
\end{Theorem}
That such curves exist can be deduced from Poonen's Bertini theorem over finite fields. They are constructed as global complete intersections of high degree (see  \cite[C.2]{Drinfeld}).  We remark:
\begin{Proposition}\label{prop:drinfeld}
\cref{thm:drinfeld} remains true over any field $k$, and there exists $Y\subset X$ satisfying the conditions of \cref{thm:drinfeld}.
\end{Proposition}
\begin{proof}

The data $(X, D, \Sigma)$ are defined over a ring of finite type $R$ over $\Z$, say $(X_R, D_R, \Sigma_R)$ such that for any closed point $s\in {\rm Spec}(R)$, the restriction $(X_s, D_s, \Sigma_s)$ fulfills the assumptions of Theorem~\ref{thm:drinfeld}. Fix such an $s$, and a $Y_s$ as in the theorem.  The equations of $Y_s$ lift to an open in $ {\rm Spec}(R)$ containing $s$. Shrinking ${\rm Spec}(R)$,  the lift $Y_R$ intersects $D_R$ in $D_R\setminus \Sigma_R$ and is transversal to $D_R\setminus \Sigma_R$, thus $Y:=Y_R\otimes_R k$ intersects $D$ in $D\setminus \Sigma$ and is transversal to $D\setminus \Sigma$, and for all closed points $t\in {\rm Spec}(R)$, $Y_t$ intersects $D_t$ in $D_t\setminus \Sigma_t$ and is transversal to $D_t\setminus \Sigma_t$.

Now let $h: V\to X\setminus D$ be a connected  finite \'etale cover, tamely ramified  along  $D\setminus \Sigma$. Writing $W:=V\times_{X\setminus D} (Y\setminus D)$, our goal is to prove that $W$ is connected.  Let $h':V'\rightarrow X$ be the normalization of $X $ in $V$, and let $g':W'\rightarrow Y$ be the normalization of $Y$ in $W$. By assumption $D\cap Y$ is finite \'etale over $k$, so we can write $D \cap Y=\coprod_{i=1}^n \Spec(k(x_i))$ with $k\subset k(x_i)$ finite separable. As $h'$ is tamely ramified with respect to $D\setminus \Sigma$, according to Abhyankar's Lemma (\cite[Cor.~2.3.4, p.~39]{Grothendieck/Tame})  there are affine \'etale neighborhoods $\eta_i:U_i\rightarrow X\setminus \Sigma$ of $x_i, i=1,\ldots, n$, such that for every $i$,  $\eta_i\times h':U_i\times_{X\setminus \Sigma} V'\rightarrow U_i$ is isomorphic to a disjoint union of Kummer coverings; we have a diagram 
\begin{equation}
\label{eq:kummer-isos}
	\begin{tikzcd}
		U_i\times_{X\setminus\Sigma} V'\ar[swap]{dr}{\eta_i\times h'}\ar{rr}{\cong}&&\coprod_{j=1}\Spec(\mathcal{O}_{U_i}[T]/(T^{e_{ij}}-a_{ij}))\ar{dl}{\text{Kummer}}\\
		&U_i
	\end{tikzcd}
\end{equation}
where the  $e_{ij}$ are prime to $\Char(k)$, the $a_{ij}\in H^0(U_i,\mathcal{O}_{U_i})$ are regular and units outside of $(D\setminus \Sigma)\times_{(X\setminus \Sigma)} U_i$.

Shrinking ${\rm Spec}(R)$, the data $(X,\Sigma, Y, D,h,h',g', \eta_i, a_{ij})$ and the isomorphisms from \eqref{eq:kummer-isos} are defined over  
$R$; denote by $(X_R, \Sigma_R, \ldots)$ the corresponding models over $R$.  Shrinking $\Spec R$ again, we may assume that $h'_R: V'_R\rightarrow X_R\setminus \Sigma_R$ is \'etale over $X_R\setminus D_R$, that $g'_R:W'_R\rightarrow Y_R$ is \'etale over $Y_R\setminus D_R$ and that $Y_s$ is  geometrically irreducible for all closed points $s\in {\rm Spec}(R)$. 

Moreover, as $D\setminus \Sigma$ is smooth and as $Y$ intersects $D$ transversally and in  $D\setminus \Sigma$, we may assume that $D_R\cap Y_R$ is finite \'etale over $\Spec R$, and that

\[
	\coprod_i \eta_{i,R}|_{Y_R\cap D_R}:\coprod_iU_{i,R}\times_{X_R} (Y_R\cap D_R) \rightarrow (Y_R\cap D_R)	
\] 
is surjective.  

For $s\in \Spec R$ a closed point of residue characteristic prime to the exponents $e_{ij}$ from \eqref{eq:kummer-isos}, the morphisms $\eta_{i,s}|_{Y_s}:U_{i,s}\times_{X_s} Y_s\rightarrow Y_s\setminus \Sigma_s$ are \'etale neighborhoods of the  points of $Y_s$ lying on $D_s\setminus \Sigma_s$, and each $g'_s\times \eta_{i,s}$ is isomorphic to a disjoint union of Kummer coverings. Thus, again by Abhyankar's Lemma (\cite[Cor.~2.3.4, p.~39]{Grothendieck/Tame}), $g'_s:W'_s\rightarrow Y_s\setminus \Sigma_s$ is tamely ramified along $(Y_s\cap D_s)\setminus \Sigma_s$.

The  morphism $\lambda: W_R'\to {\rm Spec}(R)$ is projective, thus shrinking ${\rm Spec}(R)$ again, one has base change for $\lambda_*\mathcal{O}_{W'_R}$. By \cref{thm:drinfeld},  $H^0(W'_s, \mathcal{O}_{W'_s})=k(s)$. Thus $\lambda_*\mathcal{O}_{W'_R}$ is a  $R$-projective module of rank $1$, thus  by base change again, $H^0(W', \mathcal{O}_{W'})=k$, thus $W$ is connected.  This finishes the proof.
\end{proof}

\begin{Remark} \label{rmk:kerz}
Recall that in \cite{Kerz/Tameness}, tame coverings of $X\setminus D$ in Theorem~\ref{thm:drinfeld} are defined, and more generally, tame coverings of regular schemes of finite type over an excellent, integral, pure-dimensional scheme.   They build a Galois category, with Galois group $\pi_1^{\rm tame}(X\setminus D, \bar y)$, which is a full subcategory of the Galois category of the covers considered by Drinfeld in \cref{thm:drinfeld}, where he considered the tameness condition only along $D\setminus \Sigma$. Thus Proposition~\ref{prop:drinfeld} implies that the functoriality homomorphism 
$\pi_1^{\rm tame}(Y \setminus D\cap Y, \bar y)\to \pi_1^{\rm tame}(X\setminus D, \bar y) $ is surjective. 
As in \cref{rmk:char0}, we observe that this latter formulation in characteristic $0$ follows from 
\cite[1.2, Remarks, p.~153]{GoreskyMacPherson}.
\end{Remark}

\section{Comments}\label{sec:generalization}
Drinfeld's theorem, unlike \cref{thm:main},  does not request $X$ to be a good compactification of $X\setminus D$. In view of the lack of resolution of singularities, this is the best possible formulation.  It is not clear how to formulate a version of \cref{thm:main}, \ref{main:LEFF}  without having a good compactification. 

\medskip

%
Let $X$ be a smooth projective, connected $k$-scheme, and let $D$  be a strict normal crossings divisor.  If $k$ is perfect, in \cite{Kerz/Lefschetz}, a quotient $\pi_1^{\ab}(X,D)$ of $\pi_1^{\et,\ab}(X \setminus D)$ is defined. There are canonical quotient homomorphisms
\begin{equation*}
\begin{tikzcd}
	\pi_1^{\et,\ab}(X\setminus D)\ar[two heads]{r}\ar[two heads]{d}&\pi_1^{\et,\ab}(X,D)\ar[two heads]{d}\\
	\pi_1^{\tame,\ab}(X \setminus D)\rar[-, double equal sign distance]&\pi_1^{\et,\ab}(X,D_{\red}),
\end{tikzcd}
\end{equation*}
where the groups in the left column are the abelianizations of the \'etale and tame fundamental group.  Let $\ell$ a prime number different from $\Char(k)$. The $\overline{{\Q}}_{\ell}$-lisse sheaves of rank $1$, which have ramification bounded by $D$ in the sense of  \cite[Def.~3.6]{Esnault/Kerz}, are precisely the irreducible $\ell$-adic representations of $\pi_1^{\et,\ab}(\bar{X},D)$.  The main result of \cite{Kerz/Lefschetz} is a Lefschetz theorem in the form of \cref{thm:main} for $\pi_1^{\et,\ab}(\bar{X},D)$.	
		
\medskip

One would wish to have a general notion of fundamental group $\pi_1^{\et}(X,D)$ encoding finite  \'etale covers with ramification bounded by $D$, and to show a Lefschetz theorem of the kind \cref{thm:drinfeld} for them. This would shed a new light  on Deligne's finiteness theorem \cite{Esnault/Kerz} over a finite field. 

\bibliographystyle{amsalphacustomlabels}
\bibliography{alggeo}

\end{document}